\documentclass[10pt,a4paper,reqno]{amsart}
\usepackage{amssymb,amsmath,amsthm}
\usepackage{xcolor}
\usepackage[linktocpage,colorlinks=true,linkcolor=teal,citecolor=magenta,breaklinks=true]{hyperref}
\usepackage{type1cm}%
\usepackage[shortlabels]{enumitem}
\setlist{font=\normalfont\upshape}
\usepackage[all,cmtip]{xy}
\theoremstyle{plain}
\newtheorem{thm}{Theorem}[section]
\newtheorem{prop}[thm]{Proposition}
\newtheorem{lem}[thm]{Lemma}
\newtheorem{cor}[thm]{Corollary}
\newtheorem*{blankit}{Theorem}
\theoremstyle{definition}

\newtheorem{rem}[thm]{Remark}

\makeatletter
\def\longrightleftarrows{\mathrel{
	\mathop{\vcenter{
	\offinterlineskip
	\hbox to 0.6truecm{\rightarrowfill}%
	\hbox to 0.6truecm{\leftarrowfill}}}%
	}}

\newcommand{\wtilde}[2][3]{%
	{}\mkern#1mu\widetilde{\mkern-#1mu#2}}
\newcommand{\what}[2][3]{%
	{}\mkern#1mu\widehat{\mkern-#1mu#2}}

\newdimen\ex@
\def\nointerlineskip{\prevdepth-\@m\p@}
\def\@projlim{%
		\mathop{\vtop{\ialign{##\crcr
		\hfil\rm lim\hfil\crcr\noalign{\nointerlineskip}\leftarrowfill\crcr
		\noalign{\nointerlineskip\kern-\ex@}\crcr}}}
}
\def\@injlim{%
		\mathop{\vtop{\ialign{##\crcr
		\hfil\rm lim\hfil\crcr\noalign{\nointerlineskip}\rightarrowfill\crcr
		\noalign{\nointerlineskip\kern-\ex@}\crcr}}}
}
\def\varprojlim{\mathop{\@projlim}}
\def\varinjlim{\mathop{\@injlim}}

\def\subsection{\@startsection{subsection}{2}%
 \z@{.5\linespacing\@plus.7\linespacing}{-.5em}%
 {\bfseries\mathversion{bold}}}
\makeatother

\let\cal=\mathcal
\def\N{\mathbb{N}}
\def\Z{\mathbb{Z}}
\def\Q{\mathbb{Q}}
\def\R{\mathbb{R}}
\def\id{\operatorname{id}}
\def\ker{\operatorname{Ker}}
\def\coker{\operatorname{Coker}}

\let\im=\Im\relax
\def\Hom{\operatorname{Hom}}
\def\End{\operatorname{End}}
\def\Ext{\operatorname{Ext}}
\def\Cogen{\operatorname{Cogen}}

\def\Rej{\operatorname{Rej}}
\def\rMod{\textup{Mod}\mathchar`-}
\def\lMod{\mathchar`-\textup{Mod}}
\def\rmod{\textup{mod}\mathchar`-}
\def\lmod{\mathchar`-\textup{mod}}

\def\De{\mathcal{D}}
\def\Add{\operatorname{Add}}

\def\chb{{\textstyle\boldsymbol{\cdot}}}
\let\rho\varrho\let\Gamma\varGamma\let\Delta\varDelta
\let\Lambda\varLambda\let\Sigma\varSigma\relax

\title{Cotilting duality for artinian rings}%

\author[F.~Mantese]{Francesca Mantese}
\address[Francesca Mantese]{%
Dipartimento di Informatica -- Settore di Matematica\\
Universit\`a di Verona\\
Strada le Grazie 15 -- Ca' Vignal, I-37134 Verona, Italy}
\author[L.~Martini]{Lorenzo Martini}
\email{francesca.mantese@univr.it}
\email{lorenzo.martini@univr.it}


\begin{document}\fontsize{10.5pt}{13.5pt}\selectfont%
\begin{abstract}
A classical result due to Morita and Azumaya establishes that given two arbitrary rings, any duality between their finitely generated modules is representable by a faithfully balanced bimodule which is a finitely generated injective cogenerator of both rings and, equivalently, these latter are one-sided artinian. We extend this well-known result to the case of a cotilting bimodule, by analysing the duality it represents in the bounded derived categories of the given rings.
\end{abstract}
\maketitle

\section{Introduction}%
It is well known that, in contrast to (additive) equivalences, no duality can affect the whole module categories $S\lMod$ and $\rMod R$ over two given rings. Nonetheless, given two subcategories ${}_S\mathcal{M}$ and $\mathcal{M}_R$, the somehow reasonable conditions ${}_SS\in{}_S\mathcal{M}$ and $R_R\in\mathcal{M}_R$ for a duality $F:{}_S\mathcal{M}\rightleftarrows\mathcal{M}_R:G$ are enough to make it representable by a bimodule ${}_SU_R$; that is, there are natural isomorphisms $F\cong\Delta_S$ and $G\cong\Delta_R$, where each $\Delta$ is the contravariant Hom functor of $U$. In particular, the modules involved in such a duality are $U$-reflexive, meaning that for all $M\in\mathcal{M}$ the evaluation maps
\begin{align*}
	\omega_M\colon M &\longrightarrow \Delta^{2}(M) \cr 
	x &\longmapsto
		\begin{aligned}[t]
			\wtilde x\colon \Delta(M) &\to U \cr
			f &\mapsto f(x)
		\end{aligned}
\end{align*}
are isomorphisms. Classical Morita theory of duality characterises the representable dualities among thick subcategories of modules (i.e.\ subcategories with the three-out-of-two condition on short exact sequences). In details, in a duality as above the involved subcategories are thick if, and only if, the representing bimodule is faithfully balanced and an injective cogenerator on both sides. These bimodules are called Morita bimodules. In general, the structure of the rings admitting a Morita bimodule $U$, yet that of $U$-reflexive modules, is not well understood. Nonetheless, a crucial result by Azumaya and Morita characterises the one-sided artinian rings admitting a Morita bimodule $U$ as the rings over which the $U$-reflexive modules are precisely the finitely generated ones. We gather this fact in the following statements.

\begin{thm}[{\cite[Theorem~6]{Azu59}}]
Let $R$ be a right artinian ring, $U_R$ a finitely generated injective cogenerator of $\rMod R$, and $S=\End_R(U)$. Then:%
\label{t:Azumaya1}%
\begin{enumerate}[(i)]
\item ${}_SU_R$ is a Morita bimodule;
\item $S$ is a left artinian ring;
\item $\Delta_R$ induces a duality between the finitely generated right $R\mathchar`-$ and left $S$-modules.
\end{enumerate}
\end{thm}

\begin{thm}[{\cite[Theorem~10]{Azu59}}]
Let $R$ and $S$ be two rings, and $\rmod R$ resp.\ $S\lmod$ be the subcategories of finitely generated right $R\mathchar`-$ resp.\ left $S$-modules. Given a duality $F: S\lmod\rightleftarrows\rmod R:G$, then:%
\label{t:Azumaya2}%
\begin{enumerate}[(i)]
\item $F$ and $G$ are both representable by a Morita bimodule ${}_SU_R$ which is finitely generated on both sides;
\item $R$ is right artinian and $S$ is left artinian.
\end{enumerate}
\end{thm}

Notice that, \emph{a priori}, $\rmod R$ and $S\lmod$ are not thick subcategories; they turn out to be so since the rings are proven to be one-sided artinian.

Referring the reader to the monograph \cite{Xue92} for an exhaustive treatment of Morita duality, we explicitly state the following consequences to the previous theorems, which we will be particularly interested in.

\begin{cor}
Let $R$ be a right artinian ring. Then $R$ has a Morita duality if and only if it admits a finitely generated injective cogenerator.%
\label{c:Azumaya1}
\end{cor}

\begin{cor}
Let $R$ be a right artinian ring. Given a Morita bimodule ${}_SU_R$, the $U$-reflexive modules are the finitely generated ones.%
\label{c:Azumaya2}%
\end{cor}

Morita theory has been generalised in a very broad sense, mostly since the advent of tilting theory in the eighties---nowadays playing a prominent role in distinguished categorical contexts, such as derived and triangulated categories (see e.g.\ \cite{MV18,NSZ18}). From a module-theoretic point of view, the definition of Morita bimodule---formerly, that of injective cogenerator---is generalised by the notion of $1$-cotilting module (namely, of a module which is injective with respect to the class of modules it cogenerates); in particular, a \emph{cotilting bimodule} is defined as a faithfully balanced bimodule which is $1$-cotiling on both sides. We refer the reader to subsection~\ref{SS:CotiltingBimodules} to any detail on cotiling (bi)modules. It is then natural to ask how cotilting bimodules shall generalise Morita bimodules, either by looking at their reflexive modules and understanding whether they represent some duality. In this direction, a vital role is played by the fact that any $1$-cotilting module cogenerates a non-trivial torsion pair, unlike any injective cogenerator. For instance, the Cotilting theorem (see e.g.\ \cite[Theorem~6]{Col98}) ensures that a cotilting bimodule ${}_SU_R$ yields two pairs of dualities, via its contravariant $\Hom$ and $\Ext^1$ functors, among suitable subcategories of the kernel of these functors, which indeed form two torsion pairs. Comparing this with Morita duality, one remarkable feature that is missing is a representability criterion; still, the classes involved in both dualities are far from being classified. In fact, representability results extending Morita theory have been proven to be more likely over derived module categories, still by generalised tilting theory, namely by means of distinguished complexes. For instance, in \cite{Miy98} a notion of cotilting bimodule complex is introduced, and proved to representing---over one-sided coherent rings---derived counterparts of Morita dualities. However, no remarkable feature of Azumaya's Morita theory seems to be decodable by such a context.

In this paper, we exploit some crucial results achieved in \cite{MT10,MT14} concerning dualities in derived categories of modules; we will prove that a cotilting bimodule induces a duality, termed \emph{cotilting duality}, for which an analogue of the aforementioned Azumaya's theorems on Morita duality holds, when finitely generated modules are replaced by bounded complexes with finitely generated cohomology. In particular, on the one hand we will get a criterion for a one-sided artinian ring to admit a cotilting duality; on the other hand, we will prove a representability theorem, that reflects the necessity of both the rings to be one-sided artinian, in order to their bounded complexes with finitely generated cohomology being dual to each other.

Let us discuss more in detail the background and our results. Recall that, for any ring $R$, $\De^b(R)$ will denote the bounded derived category of $\rMod R$. Given a cotilting bimodule ${}_SU_R$, we call \emph{cotilting duality} the duality induced by the right adjunction
\[
	\mathbf{R}\Delta_S:\De^b(S^{\rm op})\longrightleftarrows\De^b(R):\mathbf{R}\Delta_R
\]
between the right derived functors of $\Hom_S(-,U)$ and $\Hom_R(-,U)$ (following \cite[Example~4.6]{MT10}, this setup encompasses the dualities induced by the $\Hom$ and $\Ext^1$ functors as described by the Cotilting theorem). In other words, a cotilting duality is the duality among the bounded complexes $M^\chb$ whose counit maps $\what\omega_M\colon M^\chb\to\mathbf{R}\Delta\circ\mathbf{R}\Delta(M^\chb)\cong\mathbf{L}\Delta^{2}(M^\chb)$ are quasi-isomorphisms. We will refer to these complexes as the \emph{$\De$-reflexive} complexes, and denote by $\De_R$ and $\De_{S^{\rm op}}$ their relevant classes. These classes are triangulated subcategories of $\De^b(R)$ and $\De^b(S^{\rm op})$, respectively. A module is called $\De$-reflexive in case it is $\De$-reflexive as a stalk complex; in \cite[Theorem~3.10]{MT10} it is proved that the subcategories $\mathcal{M}_R$ and ${}_S\mathcal{M}$ of $\De$-reflexive modules are exact abelian subcategories of $\rMod R$ and $S\lMod$, respectively. In \cite[Corollary~3.6]{MT10}, it is proved that a complex is $\De$-reflexive if and only if its cohomology modules are $\De$-reflexive. Consequently, in the notation we set there are equivalences of triangulated categories
\[
	\De_R^{\vphantom b}\cong\De^b_{\mathcal{M}_R}(R)\quad\hbox{and}\quad
	\De_{S^{\rm op}}^{\vphantom b}\cong \De^b_{{}_S\mathcal{M}}(S^{\rm op}),
\]
where $\smash[b]{\De^b_{\mathcal{M}_R}(R)}$ denotes the bounded derived category of complexes with cohomology in $\mathcal{M}_R$.

To motivate our study of cotilting dualities as in the context of Morita theory for artinian rings, let us recall two results from \cite{MT14}. In \cite[Example~1.3]{MT14} it is exhibited a two-sided artinian ring which does not admit a Morita bimodule but admits a cotilting bimodule. In particular, cotiling bimodules strictly generalise Morita bimodules even over artinian rings. In \cite[Corollary~3.7]{MT14} it is proved that if $R$ is right artinian and ${}_SU_R$ is a cotilting bimodule, then the $\De$-reflexive right $R$-modules coincide with the finitely generated ones, i.e.\ $\mathcal{M}_R=\rmod R$. This expresses how Corollary~\ref{c:Azumaya2} extend to cotilting dualities.

This said, we have to face at once a substantial lack of symmetry with respect to the classical case: if $R$ is right artinian and ${}_SU_R$ is a cotilting bimodule, then neither $S$ needs to be left artinian (see Remark~\ref{r:CotiltingBimodulesOverArtinian}) nor the $\De$-reflexive left $S$-modules have to be finitely generated. We prove in Theorem~\ref{t:CotiltingBimoduleBothArtinianRings} that the occurance of one condition is in fact equivalent to the occurance of the other. Consequently, if $R$ is right artinian and $S$ is left artinian, any cotilting duality between $R$ and $S$ should appear in the shape
\[
	\mathbf{R}\Delta_S:\De^b_{S\lmod}(S^\textup{op})\longrightleftarrows
	\De^b_{\rmod R}(R):\mathbf{R}\Delta_R.
\]
Altogether, we need to understand when a right artinian ring admits a cotilting duality. In our first main result, we classify these rings still by means of a finiteness condition on their cotilting modules, namely their product completeness, getting the following incarnation of Corollary~\ref{c:Azumaya1}.

\begin{blankit}[Corollary~\ref{c:ExistenceCotiltingDuality}]
A right artinian ring $R$ has a cotilting duality if and only it admits a finitely generated product complete cotilting module.
\end{blankit}

Notice that the previous is an effective extension of Morita theory for artinian rings to cotiling dualities. Indeed, the aforementioned one-sided artinian ring provides an instance of a ring without a Morita duality and with a cotilting duality.

Our second main result is the following representability theorem, which embodies Theorems~\ref{t:Azumaya1} and~\ref{t:Azumaya2}.

\begin{blankit}[Theorem~\ref{t:Representability}]
Let $F:S\lMod\rightleftarrows\rMod R:G$ be two additive contravariant functors such that
$$
	\mathbf{R}F:\De^b_{S\lmod}(S^{\rm op})\longrightleftarrows\De^b_{\rmod R}(R):\mathbf{R}G
$$
is a duality. Assume that:%
\begin{enumerate}[(1)]
\item the cohomological dimension of $F$ and $G$ is at most~$1$;
\item $F$ and $G$ send projectives into $G$-acyclic and $F$-acyclic objects, respectively.
\end{enumerate}
Then, the following hold:
\begin{enumerate}[(i)]
\item $\mathbf{R}F$ and $\mathbf{R}G$ are both representable by a cotilting bimodule ${}_SU_R$ which is finitely generated on both sides;

\item $R$ is right artinian and $S$ is left artinian.
\end{enumerate}
\end{blankit}

Let us briefly motivate the stated hypotheses. They make \cite[Corollary~3.6]{MT10} apply, so that the $\De$-reflexive complexes of the duality are the bounded complexes with $\De$-reflexive cohomology modules; in particular, (2)~is equivalent to the composition of the right derived functors of $F$ and $G$ being the left derived functor of their composition (see \cite[Proposition~5.4]{Har66}). Notice that, analogously to Theorem~\ref{t:Azumaya2}, a priori $\rmod R$ and $S\lmod$ are not exact abelian subcategories, though they turn out to be so for $R$ and $S$ being one-sided artinian, so that the derived categories $\De^b_{\rmod R}(R)$ and $\De^b_{S\lmod}(S^\textup{op})$ exist and are triangle equivalent to $\De^b(\rmod R)$ and $\De^b(S\lmod)$, respectively.

The paper is organised as follows. In Section~\ref{S:Preliminaries} we recall the notion of $\De$-reflexivity, specifically focusing on that related to a cotilting duality. In Section~\ref{S:OneSidedArtinianRing} we study for a right artinian ring $R$ the cotilting duality induced by a cotilting bimodule ${}_SU_R$: we prove that $U_R$ is product complete and, conversely, that $R$ has a cotilting duality if and only if it admits a finitely generated product complete cotilting module $U_R$. Subsequently, in Section~\ref{S:DReflexiveModules} we analyse all the properties of Azumaya's Morita duality. On the one hand, for a right artinian ring $R$ and a cotilting bimodule ${}_SU_R$ we characterise when $S$ is left artinian and the $\De$-reflexive left $S$-modules are finitely generated; on the other hand, by no assumption on $R$ and $S$, we identify the dualities among their bounded derived categories of complexes with finitely generated cohomology which are representable by a cotilting bimodule ${}_SU_R$, proving that the rings are necessarily one-sided artinian.

\section{Preliminaries}%
\label{S:Preliminaries}%
All the rings we consider are associative with $1\ne0$. Given a ring $R$, $\rMod R$ (resp.\ $R\lMod$) will denote the category of its right (resp.\ left) $R$-modules; $\rmod R$ (resp.\ $R\lmod$) the subcategory of the finitely generated right (resp.\ left) $R$-modules; and $\De(R)$ (resp.\ $\De^b(R)$) the unbounded (resp.\ bounded) derived category of $\rMod R$. Given a right $R$-module $C$, we set
\begin{align*}
	{}^{\bot_0}C &= \{M\in\rMod R\mid \Hom_R(M,C)=0\}, \cr
	{}^{\bot_1}C &= \{M\in\rMod R\mid \Ext^1_R(M,C)=0\},
\end{align*}
and we will denote by $\operatorname{Cogen}C$ the class of modules cogenerated by $C$, i.e.~those modules fitting in an exact sequence $0\to M\to C^\alpha$, for some cardinal $\alpha$.

\subsection{Reflexivity and \texorpdfstring{$\cal D$}{D}-reflexivity}%
Given a right adjoint pair of contravariant additive functors $F:\mathcal{A}\rightleftarrows\mathcal{B}:G$ between two abelian categories, one defines its \emph{reflexive objects} as those objects whose counit morphisms are isomorphisms. For instance, as mentioned in the introduction, over module categories dualities among thick subcategories are essentially represented by Morita bimodules $U$, and the reflexive modules are $U$-reflexive.

It is proved in \cite[Lemma~13.6]{Kel96} that the right adjunction 
$F:\mathcal{A}\rightleftarrows\mathcal{B}:G$ extends to a right adjunction $\mathbf{R}F:\De^b(\mathcal{A})\rightleftarrows\De^b(\mathcal{B}):\mathbf{R}G$, as long as the bounded derived categories of $\mathcal{A,B}$ and the right derived functors of $F,G$ exist. As shown in \cite[Example~2.4]{MT10}, one gets two independent notions of reflexivity within the underlying abelian categories, namely when looking either at the objects involved in the adjunction $(F,G)$ or at the objects that, as stalk complexes, are involved in the derived adjunction $(\mathbf{R}F,\mathbf{R}G)$. Thus, the latter yet all the reflexive complexes shall be called \emph{$\De$-reflexive} instead. In \cite[Corollary~3.6]{MT10} it is proved that, under certain homological conditions, a complex is $\De$-reflexive if, and only if, its cohomologies are $\De$-reflexive modules. As we will recall in the forthcoming Remark~\ref{r:CotiltingBimodules}, such conditions are fulfilled by the adjunction we are going to focus on in the present paper, namely the one provided by the contravariant Hom functors of a $1$-cotilting bimodule.

\subsection{Cotilting bimodules}%
\label{SS:CotiltingBimodules}%
Let $R$ be a ring and $n\in\N$. Following \cite[Definition~15.1]{GT12}, a right $R$-module $U$ is called $n$-cotilting if
\begin{enumerate}[(i)]
\item the injective dimension of $U_R$ is $\empty\le n$;
\item $\Ext^i_R(U^\alpha,U)=0$, for all $1\le i\le n$ and every cardinal $\alpha$;
\item there exist $r\ge0$ and an exact sequence $0\to U_r\to\cdots\to U_1\to U_0\to Q\to0$, where each $U_i$ is a direct summand of a product of copies of $U$, and $Q$ is an injective cogenerator of $\rMod R$.
\end{enumerate}

\begin{rem}
Let us recall further well-known facts on $1$-cotilting modules. Given a ring $R$, a module $U_R$ is $1$-cotilting if and only if ${}^{\bot_1}U=\operatorname{Cogen}U$; if and only if ${}^{\bot_0}U\cap{}^{\bot_1}U=\{0\}$ and conditions (i)--(ii) above, for $n=1$, hold. Recall that a module $U_R$ satisfying conditions (i)--(ii) is called \emph{partial cotilting}.%
\label{r:CotiltingModules}%

In particular, any $1$-cotilting module $U_R$ gives rise to the torsion pair $({}^{\bot_0}U,{}^{\bot_1}U)$ in $\rMod R$, meaning that every right $R$-module $M$ fits in a functorial short exact sequence $0\to X\to M\to Y\to0$, with $\Hom_R(X,U)=0$ and $\Ext^1_R(Y,U)=0$. The torsion radical of the torsion pair, i.e.~the functor $M\mapsto X$, is explicitly described as the \emph{reject in $U$}, namely
\[
	\Rej_U(M)=\bigcap\{\ker f\mid f\in\Hom_R(M,U)\},
\]
for all $M_R$ (notice that $\Rej_U(M)=\ker\omega_M$). By dropping any reference to $U_R$, we will call the modules in ${}^{\bot_0}U$ simply the \emph{torsion modules}, whereas the modules in ${}^{\bot_1}U$ will be called \emph{torsionfree modules}. Moreover, we set
\[
	\Delta_R=\Hom_R(-,U)\quad\hbox{and}\quad
	\Gamma_{\!R}=\Ext^1_R(-,U).
\]
\end{rem}

A $S$-$R$-bimodule ${}_SU_R$ is called \emph{$n$-cotilting bimodule} if it is faithfully balanced and ${}_SU$ and $U_R$ are $n$-cotilting modules. Thus, $0$-cotilting bimodules are precisely the Morita bimodules (see also \cite[Remark~15.8]{GT12}). Henceforth, we will refer to a $1$-cotilting bimodule just as a \emph{cotilting bimodule} (see Remark~\ref{r:CotiltingBimodules}).

\begin{rem}
As cotilting bimodules generalise Morita bimodules, in view of \cite[Remark~2.9]{Xue92} let us survey what happens in the context we are interested in---the artinian setting. Recall that, given any bimodule ${}_SM_R$, the triangular matrix ring
\(
	\Lambda=\bigl(\!\begin{smallmatrix} 
	R & 0 \cr
	M & S
	\end{smallmatrix}\!\bigr)
\)
is a right (left) artinian ring if and only if both $S$ and $R$ are right (left) artinian rings, and $M_R$ (${}_SM$) is finitely generated.%
\label{r:CotiltingBimodulesOverArtinian}

Now, on the one hand cotilting bimodules extend Morita bimodules over artinian rings. Indeed, let $D$ be a division ring having a division subring $C$ such that $\dim({}_CD)$ is finite but $\dim(D_C)$ is not. By the above, the triangular matrix ring
\(
	\Lambda=\bigl(\!\begin{smallmatrix} 
	D & 0 \cr
	D & C
	\end{smallmatrix}\!\bigr)
\)
is a two-sided artinian ring, but it has no Morita bimodules (for $\Lambda\lmod$ and $\rmod \Lambda$ are not dual to each other); nonetheless, ${}_\Lambda\Lambda_\Lambda$ is a cotilting bimodule by \cite[Corollary~1.2]{MT14}.

On the other hand, if $R$ is a right artinian ring and ${}_SU_R$ is a cotilting bimodule, then $S$ needs not to be left artinian. Indeed, let
\(
	\Lambda=\bigl(\!\begin{smallmatrix} 
		\R & 0 \cr
		\R & \Q
	\end{smallmatrix}\!\bigr)
\);
then $\Lambda$ is right artinian, ${}_\Lambda\Lambda_\Lambda$ is a cotilting bimodule by \cite[Corollary~1.2]{MT14} again, and $\Lambda$ is not even left noetherian. Nonetheless, $S$ is always left coherent (see the end of Remark~\ref{r:FinitenessConditions}).
\end{rem}

As announced, given a cotilting bimodule ${}_SU_R$ we are going to analyse the duality induced by the right adjunction of the right derived functors of $\Delta_S$ and $\Delta_R$, especially focusing on the case of $R$ being right artinian. Thus, we will study the duality induced by
\[
	\mathbf{R}\Delta_S:\De^b(S^{\rm op})\longrightleftarrows \De^b(R):\mathbf{R}\Delta_R
\]
which we refer to as \emph{cotilting duality} (still implying any reference to ${}_SU_R$).
When there is no risk of confusion, we shall denote by $\Delta$ either $\Delta_S$ and $\Delta_R$ (and by $\Gamma$ either $\Gamma_{\!S}$ and $\Gamma_{\!R}$).

\begin{rem}
We are going to deal exclusively with $1$-cotilting bimodules, for the duality they induce fulfils the aforementioned conditions of \cite[Corollary~3.6]{MT10}, so that a complex is $\De$-reflexive if, and only if, its cohomologies are $\De$-reflexive modules. In details, given a cotilting bimodule $U$, since its injective dimensions are at most~$1$, the functors $\Delta$'s have cohomological dimension at most~$1$ (for $R^\ast\Delta=H^\ast\mathbf{R}\Delta\cong\Ext^\ast(-,U)$), and since $\Gamma(U^\alpha)=0$ for any cardinal $\alpha$, projective modules are sent to $\Delta$-acyclic objects.%
\label{r:CotiltingBimodules}%

Notice that the previous properties also provide, for any cotilting bimodule, the natural isomorphisms $\mathbf{R}\Delta\circ\mathbf{R}\Delta\cong \mathbf{L}\Delta^{2}$ (see \cite[Proposition~5.4]{Har66}); in particular, the $\De$-reflexive objects shall also be regarded as the complexes $M^\chb$ whose counit map $\what\omega_M\colon M^\chb\to\mathbf{L}\Delta^{2}(M^\chb)$ is a quasi-isomorphism.
\end{rem}

As we just said, in a cotilting duality the $\De$-reflexive complexes are precisely the complexes whose cohomologies are $\De$-reflexive modules. We now recall the following useful general result describing the $\De$-reflexive modules.

\begin{prop}[{\cite[Proposition~1.7]{MT14}}]
Let ${}_SU_R$ be a cotilting bimodule. For every $M\in\rMod R$ there exists a commutative diagram with exact row%
\label{p:DReflexiveModules}%
\[
\xymatrix{%
	&& M \ar[d]_-{H^0(\what\omega_M)}\ar[dr]^-{\omega_M} \cr
	0 \ar[r] & \Gamma^2(M) \ar[r] & L^{0}\Delta^{2}(M) \ar[r] & \Delta^{2}(M) \ar[r] & 0
}
\]
Moreover, for all $j\in\Z$,
\[
	L^{j}\Delta^{2}(M) =
	\begin{cases}
		0 & \hbox{if $j\ne-1,0$,} \cr
		\Delta\Gamma(M) & \hbox{if $j=-1$.}
	\end{cases}
\]
In particular, $M_R$ is $\cal D$-reflexive if and only if $\Delta\Gamma(M)=0$ (whence $\mathbf{L}\Delta^{2}(M)$ is a stalk complex) and $H^0(\what\omega_M)$ is an isomorphism.
\end{prop}

Concerning the case we are more interested in, the following holds.

\begin{cor}[{\cite[Corollary~3.7]{MT14}}]
Let $R$ be a right artinian ring, and ${}_SU_R$ a cotilting bimodule. Then the $\De$-reflexive right $R$-modules coincide with the finitely generated ones. In particular, $U_R$ is of finite lenght.%
\label{c:CotiltingBimoduleOneSidedArtinian}%
\end{cor}

In Remark~\ref{r:CotiltingBimodulesOverArtinian} we saw that cotilting bimodules effectively generalise Morita bimodules over artinian rings, and that if one ring is one-sided artinian, then the other ring may not be one-sided artinian as well. By virtue of our main result (Theorem~\ref{t:ExistenceCotiltingDuality}) the last example of the remark exhibits a right artinian and not left artinian ring $\Lambda$ admitting a cotilting duality: in fact, the regular bimodule ${}_\Lambda\Lambda_\Lambda$ is cotilting, and $\Lambda_\Lambda$ is product complete.

For the sake of the reader, we recall the notion of product completeness together with some results from \cite{KS97} concerning its interplay with other finiteness conditions.

\begin{rem}
Let $R$ be a ring and $M$ a right $R$-module with $S=\End_R(M)$. The module $M_R$ is called:%
\label{r:FinitenessConditions}%
\begin{itemize}
\item \emph{endofinite}, if it is of finite length over $S$;

\item \emph{product complete}, provided the class $\operatorname{Add}M$ of the direct summands of coproducts of copies of $M$ is closed under products;

\item \emph{$\Sigma$-pure-injective}, if any coproduct of copies of $M$ is pure-injective.
\end{itemize}
Every endofinite module is product complete (these notions coincide over noetherian rings), and every product complete module is $\Sigma$-pure-injective (the notions are distinct even over $\Z$). More in details, $M$ is endofinite if and only if every direct summand of $M$ is product complete, if and only if $M$ is $\Sigma$-pure-injective and ${}_SM$ is noetherian. Eventually, a finitely generated module $M_R$ is product complete if and only if $S$ is left coherent and right perfect, and ${}_SM$ is finitely presented.

In case $R$ is right artinian and $M$ is a cotilting bimodule ${}_SU_R$, then by Corollary~\ref{c:CotiltingBimoduleOneSidedArtinian} it follows that ${}_SU$ is endofinite; by \cite[Corollary~5.4]{CMT10} we have that $S$ is left coherent.
\end{rem}

\section{When \texorpdfstring{$R$}{R} is right artinian}%
\label{S:OneSidedArtinianRing}%
As we discussed above, if ${}_SU_R$ is a cotilting bimodule and $R$ is a right artinian ring, then $U_R$ is of finite length, and $S$ is left coherent. We are going to prove that $U_R$ is also product complete. Let us denote by $\mathcal{F}$ the subcategory consisting of the finitely presented left $S$-modules cogenerated by ${}_SU$. We need the following lemma, which is a specialisation of the Reiten--Ringel Condition \cite[Condition~5.1]{CMT10}. 

\begin{lem}
Let ${}_SU$ be a $\Sigma$-pure-injective cotilting module. If $0\to Y_0\to Y\to X\to0$ is a short exact sequence in $S\lMod$ with $Y_0,Y\in{}^{\bot_1}U$, $Y_0$ finitely generated, and $X\in{}^{\bot_0}U$, then $Y$ is finitely presented.
\label{l:ReitenRingel}%
\end{lem}
\begin{proof}
First let us prove that ${}_SX$ is finitely generated. Since ${}_SU$ is $\Sigma$-pure-injective, by \cite[Corollary~6.14]{GT12} and \cite[Corollary~4.4]{CMT10} the subcategory ${}^{\bot_1}U$ of $S\lMod$ coincides with the class of direct summands of $\mathcal{F}$-filtered modules. Thus, there exists a $\mathcal{F}$-filtered module $N$ such that $N=Y\oplus Y'$; by Hill lemma (see \cite[Theorem~2.1]{CMT10}) there exists a finitely presented module $X'$ such that $Y_0\le X'\le N$ with $N/X'$ being $\mathcal{F}$-filtered again, hence in particular in ${}^{\bot_1}U$. From the short exact sequence $\smash[t]{0\to X'\!/Y_0'\to N/Y_0\buildrel\phi\over \to N/X\to0}$,
since $X\cong Y/Y_0\le N/Y_0$ is in ${}^{\bot_0}U$ by assumption, we obtain $X\le\ker\phi$. Hence in particular $X\le_\oplus X'\!/Y_0$, and $X$ is finitely generated, as claimed. By extension-closure, $Y$ is finitely generated, and in particular it is a finitely generated submodule of the $\mathcal{F}$-filtered module $N$. Then, by Hill lemma, there exists a finitely presented module $X''$ such that $Y\le X''\le N$. Since the ring $S$ is left coherent, we get that $Y$ is finitely presented. \qedhere
\end{proof}

\begin{prop}
Let $R$ be a right artinian ring and ${}_SU_R$ be a cotilting bimodule. Then $U_R$ is product complete.
\label{p:ProductComplete}%
\end{prop}
\begin{proof}
By Remark~\ref{r:FinitenessConditions}, in order to prove that $U_R$ is product complete it suffices to show that $S$ is right perfect and ${}_SU$ is finitely presented. Since $U_R$ is of finite length, the ring $S$ is semiprimary and so left and right perfect (see \cite[Corollary~28.8]{AF92}). Since $R_R$ is artinian and hence finitely cogenerated, given an inclusion $R_R\lhook\joinrel\to U^\alpha$ we obtain a short exact sequence $0\to R_R\to U^n\to M\to0$ for a suitable finitely generated right $R$-module $M$ and $n\in\N$. By applying $\Delta_R$, the exact sequence $0\to\Delta(M)\to S^n\to U\to\Gamma(M)\to0$ in $S\lMod$ yields a short exact sequence $0\to K\to U\to\Gamma(M)\to0$ with $K$ a finitely generated left $S$-module. Since $M_R$ is $\cal D$-reflexive by Corollary~\ref{c:CotiltingBimoduleOneSidedArtinian}, we have $\Delta\Gamma(M)=0$, i.e.\ $\Gamma(M)\in{}^{\bot_0}U$. Now, by Lemma~\ref{l:ReitenRingel}, we conclude that ${}_SU$ is finitely presented. \qedhere
\end{proof}

Corollary~\ref{c:CotiltingBimoduleOneSidedArtinian} and Proposition~\ref{p:ProductComplete} show that if $R$ is right artinian, the existence of a cotilting bimodule ${}_SU_R$ forces $U_R$ to be finitely generated and product complete. We now state that also the converse holds true.

\begin{thm}
Let $R$ be a right artinian ring and $U_R$ a finitely generated cotilting module with $S=\End_R(U)$. If $U_R$ is product complete, then ${}_SU_R$ is a cotilting bimodule. In particular, $S$ is left coherent and semiprimary, and ${}_SU$ is finitely presented and product complete.
\label{t:ExistenceCotiltingDuality}%
\end{thm}

Altogether, we obtain the announced generalisation of Azumaya's duality.

\begin{cor}
A right artinian ring has a cotiling duality if and only if it admits a finitely generated product complete cotilting module.
\label{c:ExistenceCotiltingDuality}%
\end{cor}

Concerning the proof of Theorem~\ref{t:ExistenceCotiltingDuality}, we already know by Remark~\ref{r:FinitenessConditions} that the endomorphism ring $S$ is left coherent and ${}_SU$ is finitely presented and product complete. So, we need to prove that ${}_SU_R$ is a faithfully balanced bimodule and that ${}_SU$ is a cotilting module. We gather these proofs in several steps. First, recall that by \cite[Propositions~1.2, 2.2]{AH00} the functor $\Delta_R$ carries finitely generated $R$-modules to finitely presented left $S$-modules. Moreover, the class $\Add U_R$ is covariantly finite, meaning that any right $R$-module admits a minimal left $\Add U$-approximation; in particular, any torsionfree module $M_R$ fits in a short exact sequence $0\to M\to U^{(\alpha)}\to K\to 0$, where $K\in{}^{\bot_1}U_R$. Eventually, throughout the proof unadorned classes and functors such as $\Cogen U$, ${}^{\bot_0}U,{}^{\bot_1}U$, and $\Delta,\Gamma,\Rej_U$ will be meant to be related to ${}_SU$.

\paragraph*{\emph{Step one: the bimodule ${}_SU_R$ is faithful balanced.}}%
Since $R_R$ is finitely cogenerated by $U_R$, it admits a coresolution
\[
	0\longrightarrow R_R\buildrel f_0\over\longrightarrow U_0
	\buildrel f_1\over\longrightarrow U_1
	\buildrel f_2\over\longrightarrow\cdots
\]
where $U_i\in\operatorname{add}U_R$ and $\coker f_i\in{}^{\bot_1}U_R$ ($\operatorname{add}U_R$ is the class of the direct summands of finite coproducts of $U_R$). Thus, by \cite[Proposition~1]{Wak88} we conclude that $U$ is a faithfully balanced $S$-$R$-bimodule and $\Ext^n_S(U,U)=0$ for all $n>0$.

\paragraph*{\emph{Step two: $\Cogen U\subseteq{}^{\bot_1}U$.}}%
Assume $M$ is a torsionfree right $R$-module finitely presented by $U_R$. In particular, $M$ is finitely cogenerated by $U_R$ and so it admits a minimal left $\Add U$-approximation $0\to M\to U^n\to L\to 0$, with $L\in{}^{\bot_1}U_R$. By \cite[Lemma~1.1]{MT04}, we get that $M$ is reflexive and so, by \cite[Theorem~1.4]{MT04}, we obtain that 
\[
	\Cogen U\cap S\lmod \subseteq{}^{\bot_1}U.
\]
Since $U$ is pure-injective, once an arbitrary $N\in\Cogen U$ is expressed as the direct limit of its finitely generated submodules, by the above display $N$ is the direct limit of modules in ${}^{\bot_1}U$ and hence
\[
	\Cogen U\subseteq{}^{\bot_1}U
\]
by pure-injectivity.

\paragraph*{\emph{Step three: the injective dimension of ${}_SU$ is at most~$1$.}}%
Let $0\to K\to S^{(\alpha)}\to N\to0$ be a short exact sequence in $S\lMod$. Since ${}_SS$ is cogenerated by $U$, also $K$ is, so that $\Ext^2_S(N,U)\cong\Ext^1_S(K,U)=0$ by the previous step.

By steps two and three, we get that ${}_SU$ is partial cotilting, so that $({}^{\bot_0}U,\Cogen U)$ is a torsion pair is $S\lMod$ (see \cite[Lemma~2.6]{CTT97}). In order to prove the final step---where we will show that ${}^{\bot_1}U\subseteq\Cogen U$---we need the following lemmata. Recall that we denoted by $\mathcal{F}$ the subcategory of the finitely presented left $S$-modules cogenerated by $U$. Now set $\mathcal{T}$ to be the subcategory of finitely presented left $S$-modules in ${}^{\bot_0}U$.

\begin{lem}
If $M$ is a finitely generated torsionfree left $S$-module, then it is reflexive and finitely presented. Moreover, $(\mathcal{T,F})$ is a torsion pair in the abelian subcategory of the finitely presented left $S$-modules.%
\label{l:TorsionPair}%
\end{lem}
\begin{proof}
Let $M$ be a finitely generated torsionfree left $S$-module. By \cite[Lem\-ma~1.1]{MT04} $M$ is reflexive. Since $\Delta(M)$ is finitely cogenerated by $U_R$, it is finitely generated and so $M\cong \Delta^{2}(M)$ is finitely presented.

Consider the short exact sequence $0\to \Rej_U(M)\to M\to M/{\Rej_U(M)}\to 0$, where $M$ is finitely presented. Then $M/{\Rej_U(M)}$ is finitely generated and torsionfree, whence finitely presented by the previous part. Since $S$ is left coherent, $\Rej_U(M)$ is finitely presented. This proves that $(\mathcal{T,F})$ is the torsion pair associated to the restriction of $\Rej_U$ to the finitely presented left $S$-modules. \qedhere
\end{proof}

\begin{lem}
For any finitely generated module $N\in S\lmod$ such that $\Delta(N)=0=\Gamma(N)$, it is $N=0$.%
\label{l:Vanishing}%
\end{lem}
\begin{proof}
Suppose $N\ne0$ and $\Delta(N)=0=\Gamma(N)$, and let us conclude with a contradiction. Let $0\to K\to S^n\to N\to0$ be a short exact sequence in $S\lMod$, $n\in\N$. The module $K$ is the direct union of its finitely generated submodules $K_i$'s, say these latter directed by the poset $(I,\le)$, where $i\le j$ iff $K_i\subseteq K_j$. Then the $S^n\!/K_i$'s form a direct system of modules whose direct limit is $N$ \cite[Lemma~8.7]{Pop73}. Notice that any $K_i$ is finitely generated torsionfree, and so reflexive. For any $i\le j$ in $I$, we have the following commutative diagrams with exact rows:
\begin{gather*}
\xymatrix{%
	0 \ar[r] & K_i \ar@{>->}[d]\ar[r] &
		K \ar@{=}[d]\ar[r] &
			C_i \ar@{->>}[d]\ar[r] & 0 \\
	0 \ar[r] & K_j \ar[r] &
		K \ar[r] &
			C_j \ar[r] & 0
}\cr
\noalign{\hbox{and}}
\xymatrix{%
	0 \ar[r] & C_i \ar@{->>}[d]\ar[r] &
		S^n\!/K_i \ar@{->>}[d]\ar[r] &
			N \ar@{=}[d]\ar[r] & 0 \\
	0 \ar[r] & C_j \ar[r] &
		S^n\!/K_j \ar[r] &
			N \ar[r] & 0
}
\end{gather*}
Now, the $\Delta(C_i)$'s form an inverse system of submodules of the right $R$-module of finite length $\Delta(K)\cong U_R^n$, and, moreover, $\varprojlim \Delta(C_i)\cong\Delta(\varinjlim C_i)=0$. Since $R$ is right artinian, we conclude that there exists $\bar\imath\in I$ such that $\Delta(C_i)=0$ for all $i\ge\bar\imath$. Therefore, we can assume that the $C_i$'s and hence the $S^n\!/K_i$'s eventually belong to ${}^{\bot_0}U$. Notice that $S^n\!/K_i\notin{}^{\bot_1}U$ for all $i\in I$, otherwise from the short exact sequence $0\to K_i\to S^n\to S^n\!/K_i\to0$ we would get $\Delta(K_i)=\Delta(S^n)$, whence $K_i\cong \Delta^{2}(K_i)=\Delta^{2}(S^n)\cong S^n$, and this would imply $S^n\!/K_i=0$, contradiction to $N=\varinjlim S^n\!/K_i$. Moreover, $\Gamma(C_i)\cong\Gamma(S^n\!/K_i)$ is torsion: indeed, $K_i$ and $S^n$ are $\cal D$-reflexive, so $S^n\!/K_i$ is $\cal D$-reflexive and hence $\Delta\Gamma(S^n\!/K_i)=0$.

Now, for each $i\le j\le k$ in $I$, set $K_{ji}=K_j/K_i$ and consider the commutative diagram with exact row
\[
\xymatrix{%
	0 \ar[r] & K_i \ar@{=}[d]\ar[r] &
		K_j \ar@{>->}[d]\ar[r] &
			K_{ji} \ar@{>->}[d]\ar[r] & 0 \\
	0 \ar[r] & K_i \ar[r] &
		K_k \ar[r] &
			K_{ki} \ar[r] & 0
}	
\]
from which, by applying $\Delta$, we obtain in $\rMod R$ the commutative diagram with exact rows
\[
\xymatrix@C.5pc@R1pc{
	0 \ar[rr] && \Delta(K_{ji}) \ar[rr] && \Delta(K_j)
		\ar@{->>}[dr]\ar[rr] && 
		\Delta(K_i) \ar[rr] && \Gamma(K_{ji}) \ar[rr] && 0 \\
	& \hphantom{X_{ij}} && \hphantom{X_{ij}} && X_{ij} 
		\ar@{>->}[ur] && \hphantom{X_{ij}} && \hphantom{X_{ij}} \\
	0 \ar[rr] && \Delta(K_{ki}) \ar[uu]\ar[rr] && \Delta(K_k)
	 \ar[uu]\ar@{->>}[dr]\ar[rr] && 
		\Delta(K_i) \ar@{=}[uu]\ar[rr] && \Gamma(K_{ki})
		\ar@{->>}[uu]\ar[rr] && 0 \\
	&&&&& X_{ik} \ar[uu]
	\ar@{>->}[ur] 
}	
\]
Since $U$ is pure-injective, passing to the inverse limits, we get the exact sequence 
\[
	\varprojlim_{j\in I} X_{ij}\longrightarrow
	\Delta(K_i)\longrightarrow {\varprojlim_{j\in I} \Gamma(K_{ji})}\cong
	\Gamma(\varinjlim_{j\in I} K_{ji})\cong \Gamma(C_i) \longrightarrow 0.
\] 
Since $\Delta(K_i)$ is of finite length, it is $\varprojlim_{j\in I} X_{ij}=\bigcap_{j\in I} X_{ij}=X_{i\bar\jmath}$, for a suitable index $\bar\jmath\in I$, whence $\Gamma(C_i)\cong \Gamma(K_{\bar\jmath i})$. Moreover, since both $C_i$ and $C_{\bar\jmath}$ are in ${}^{\bot_0}U$, from $0\to K_{\bar\jmath i}\to C_i\to C_{\bar\jmath}\to0$ we get the exact sequence
\[
	0\longrightarrow \Delta(K_{\bar\jmath i})\longrightarrow \Gamma(C_{\bar\jmath})\longrightarrow
	\Gamma(C_i)\longrightarrow\Gamma(K_{\bar\jmath i})\longrightarrow0.
\]
The modules $\Gamma(C_i)$ and $\Gamma(K_{\bar\jmath i})$ have the same finite length, for they are isomorphic factors of the module of finite length $\Delta(K_i)$, so in turn $\Delta(K_{\bar\jmath i})\cong \Gamma(C_{\bar\jmath})$ by exactness. Hence, $\Gamma(C_{\bar\jmath})$ is torsionfree. Finally, since $\Gamma(C_{\bar\jmath})\cong\Gamma(S^n\!/K_{\bar\jmath})$ is torsion and $U_R$ is cotilting, then $\Gamma(S^n\!/K_{\bar\jmath})=0$, i.e.\ $S^n\!/K_{\bar\jmath}=0$. By the previous argument this is a contradiction. \qedhere
\end{proof}

\begin{lem}
Let $(Y_i)_{i\in I}$ be a monomorphic direct system of left $S$-modules in $\mathcal{F}$ such that $Y_j/Y_i\in\mathcal{T}$, for any $i\le j$. Then $\varinjlim_I Y_i\in\Cogen U$.
\end{lem}
\begin{proof}
By hypothesis on the $Y_j/Y_i$'s, applying $\Delta$ we obtain an monomorphic inverse system $(\Delta(Y_i))_{i\in I}$ in $\rMod R$, whose inverse limit is $\bigcap_{i\in I}\Delta(Y_i)$. Since $R$ is right artinian and the $\Delta(Y_i)$'s are modules of finite length, there exists $\bar\jmath\in I$ such that $\Delta(Y_i)=\Delta(Y_{\bar\jmath})$ for all $i\ge\bar\jmath$. Since each $Y_i$ is reflexive by Lemma~\ref{l:TorsionPair}, we conclude $\varinjlim_I Y_i=Y_{\bar\jmath}\in\Cogen U$. \qedhere
\end{proof}

\begin{lem}
Given a short exact sequence $0\to Y_0\to Y\to X\to0$ in $S\lMod$, with $Y_0\in\mathcal{F}$, $Y\in\varinjlim\mathcal{F}$ and $X\in\varinjlim\mathcal{T}$, then $Y\in\Cogen U$.%
\label{l:CogenCondition}%
\end{lem}
\begin{proof}
Write $X=\varinjlim_I X_i$ for the monomorphic direct system of its finitely generated submodules. Since $({}^{\bot_0}U,\Cogen U)$ is a torsion pair in $S\lMod$ and $\Hom_S(X,\varinjlim\mathcal{F})=0$, with a usual argument we may assume that actually $X_i$ is in ${}^{\bot_0}U$. Now, each $X_i$ is of the form $Y_i/Y_0$ for some $Y_i\in\mathcal{F}$ containing $Y_0$, hence we obtain the monomorphic direct system of short exact sequence $0\to Y_0\to Y_i\to X_i\to0$, in which $Y_j/Y_i\cong X_j/X_i\in\mathcal{T}$ for all $i\le j$. By the previous lemma, $Y=\varinjlim_I Y_i\in\Cogen U$. \qedhere
\end{proof}

We are now ready to conclude with the final step.

\subparagraph*{\emph{Step four: ${}^{\bot_1}U\subseteq\Cogen U$.}}%
First let us show that the inclusion holds for finitely generated modules. Let $0\to K\to S^n\to N\to0$ be a short exact sequence in $S\lMod$ with $\Gamma(N)=0$. Then $0\to\Delta(N)\to U^n_R\to \Delta(K)\to0$ is a short exact sequence of finitely generated torsionfree right $R$-modules, hence $\Delta^{2}(N)$ and $\Delta^{2}(K)$ are finitely presented modules in $\Cogen U\subseteq{}^{\bot_1}U$. Therefore, in the commutative diagram with exact rows
\[
\xymatrix{%
	0 \ar[r] & K \ar[d]_{\omega_K} \ar[r] & S^n \ar[d]^\cong \ar[r] &
		N \ar[d]^{\omega_N} \ar[r] & 0 \\
	0 \ar[r] & \Delta^{2}(K) \ar[r] & \Delta^{2}(S^n) \ar[r] & \Delta^{2}(N) \ar[r] & 0
}	
\]
we have that $\omega_N$ is epic, and $\ker\omega_N\cong\coker\omega_K$ is finitely generated. Moreover, since $\Delta^{2}(N)\in{}^{\bot_1}U$, we have $\ker\omega_N\in{}^{\bot_0}U\cap{}^{\bot_1}U$; that is, $\Rej_U(N)=0$, by Lemma~\ref{l:Vanishing}, so that $N$ is reflexive and, in particular, torsionfree and finitely presented. Observe that we proved a stronger claim than the stated one, indeed we showed that any finitely module in ${}^{\bot_1}U$ is finitely presented and torsionfree. Recall that, by \cite[(4.4), p.~1666]{CB94}, $(\varinjlim\mathcal{T},\varinjlim\mathcal{F})$ is a torsion pair in $S\lMod$. In particular, we have $\varinjlim\mathcal{T}\subseteq{}^{\bot_0}U$ and $\varinjlim\mathcal{F}={}^{\bot_1}U$ (the latter equality follows since $U$ is pure-injective and ${}^{\bot_1}U$ is closed under submodules). We shall conclude by proving that $\varinjlim\mathcal{F}\subseteq\Cogen U$. First, assume that $N\in\varinjlim\mathcal{F}\cap{}^{\bot_0}U$. For any fixed $0\ne x\in N$, we have the pullback diagram
\[
\xymatrix{%
	0 \ar[r] & Sx \ar@{=}[d]\ar[r] & Q \ar@{>->}[d]\ar[r]\ar@{}[dr]|-{\textstyle\lrcorner} &
		t(X) \ar@{>->}[d]\ar[r] & 0 \\
	0 \ar[r] & Sx \ar[r] &
		N \ar[r]\ar@{->>}[d] & X \ar[r]\ar@{->>}[d] & 0 \\
	&& Y \ar@{=}[r] & Y
}	
\]
where $t$ is the torsion radical associated with the torsion class $\varinjlim\mathcal{T}$, so that $Y\in\varinjlim\mathcal{F}$. Since $Sx\in\mathcal{F}$, $Q\in\varinjlim\mathcal{F}$ and $t(X)\in\varinjlim\mathcal{T}$, Lemma~\ref{l:CogenCondition} applies on the first row so that $Q\in\Cogen U$. In particular, there exists an $S$-linear map $f\colon Q\to U$ such that $f(x)\ne0$, and since $\Gamma(Y)=0$, it can be extended to a nonzero map $\bar f\colon N\to U$ with $\bar f(x)\ne0$. Therefore, $N\in\Cogen U$, whence $N=0$. Finally, assume $N\in\varinjlim\mathcal{F}$, and consider the short exact sequence $0\to\Rej_U(N)\to N\to N/{\Rej_U(N)}\to0$. We have $\Rej_U(N)\in\varinjlim\mathcal{F}\cap{}^{\bot_0}U$, so by the previous argument we get $\Rej_U(N)=0$ and thus $N\in\Cogen U$. \qedhere

\section{\texorpdfstring{$\cal D$}{D}-Reflexive Modules and Artinian Rings}%
\label{S:DReflexiveModules}%

As recalled in Remark~\ref{r:CotiltingBimodulesOverArtinian}, the main difference between the Morita and the cotilting settings is that, in the latter, over a right artinian ring $R$ the endomorphism ring of a finitely generated product complete cotilting module $U_R$ needs not to be left artinian. We now characterise the chain conditions of these endomorphism rings.

\begin{prop}
Let $R$ be a right artinian ring with a cotilting duality induced by $U_R$, and let $S=\End_R(U)$. The following are equivalent:
\begin{enumerate}
\item[(a)] $S$ is left artinian;

\item[(b)] $S$ is left noetherian;

\item[(c)] $\Delta\Gamma(N)=0$ for every finitely generated left $S$-module $N$.
\end{enumerate}
\end{prop}
\begin{proof}
${\rm(a)\Leftrightarrow(b)}$ The ring $S$ is semiprimary being the endomorphism ring of the finite length module $U_R$. Then $S$ is left artinian if and only if it is left noetherian.

\noindent${\rm(b)\Rightarrow(c)}$ Let $N$ be a finitely generated left $S$-module. Since $N$ is finitely presented, it is $\cal D$-reflexive by \cite[Lemma~3.1\,(1,c)]{MT14}, hence in particular $\Delta\Gamma(N)=0$ (see Proposition~\ref{p:DReflexiveModules}).

\noindent${\rm(c)\Rightarrow(b)}$ The implication follows once we prove that the finitely generated left $S$-modules are precisely the $\cal D$-reflexive ones. Indeed, given a short exact sequence $0\to I\to {}_SS\to N\to0$ in $S\lMod$, $N$ is finitely generated whence $\cal D$-reflexive by our claim; by thickness of the $\cal D$-reflexives, $I$ is is $\cal D$-reflexive whence finitely generated.

So, let us prove our claim. First notice that we are in the setting of \cite[Propositions~1.2, 2.2]{AH00}, so that the functors $\Delta$'s map finitely generated modules in finitely presented ones. Assume now that $N$ is a $\cal D$-reflexive module; the complex $\mathbf{R}\Delta_S(N)$ in $\cal D^b(R)$ is $\cal D$-reflexive, so it has $\cal D$-reflexive cohomologies. Since $R$ is right artinian, its cohomologies are finitely generated (\cite[Proposition~3.6]{MT14}), so it is quasi-isomorphic to a complex with finitely generated projective terms (see for instance \cite[Lemma~4.6]{Har66}). Then $\mathbf{R}\Delta^{2}(N)$ is a complex of finitely presented modules, whence $N\cong R^0\Delta^{2}(N)$ is finitely generated.

Conversely, let $N$ be a finitely generated $S$-module and let $0\to K\to S^n\to N\to 0$ be a $\Delta_S$-acyclic presentation of $N$. If $N$ is torsionfree, then it is $\cal D$-reflexive by \cite[Lemma~3.1\,(1,a)]{MT14}. If $N$ is torsion, by hypothesis also $\Gamma(N)$ is torsion, and we have the commutative diagram with exact rows
\[
\xymatrix{%
	0 \ar[r] & K \ar[d]_{\omega_K} \ar[r] & S^n\ar[d]^\cong \ar[r] &
		N \ar[d]^{H^0(\what\omega_N)}\ar[r] & 0 \\
	0 \ar[r] & \Delta^{2}(K) \ar[r] & \Delta^{2}(S^n) \ar[r] & \Gamma^2(N) \ar[r] & 0
}	
\]
where $H^0(\what\omega_N)$ is epic and therefore also $\Gamma^2(N)$ is torsion and finitely generated. By hypothesis we have $\Delta\Gamma(\Gamma^2(N))=0$, i.e.\ $\Gamma^3(N)$ is torsion as well. Moreover, from the following commutative diagram with exact rows,
\[
\xymatrix{%
	0 \ar[r] & U^n_R \ar[d]_\cong \ar[r] & \Delta(K)\ar@{>->}[d]^{\omega_{\Delta(K)}} \ar[r] &
		\Gamma(N) \ar[d]^{H^0(\what\omega_{\Gamma(N)})} \ar[r] & 0 \\
	0 \ar[r] & \Delta^{2}(U^n) \ar[r] & \Delta^{3}(K) \ar[r] & \Gamma^3(N) \ar[r] & 0
}	
\]
where the first one is a $\Delta_R$-acyclic presentation of $\Gamma(N)$, we obtain that $H^0(\what\omega_{\Gamma(N)})$ is monic and $\coker\omega_{\Delta(K)}\cong\coker H^0(\what\omega_{\Gamma(N)})$. Notice that $\omega_{\Delta(K)}$ is a split monomorphism, so in the previous isomorphism the first term is torsionfree while the second is torsion. Hence we get that they are both zero and in particular $\omega_{\Delta(K)}$ is an isomorphism. By adjunction it follows that also $\omega_K$ is an isomorphism, hence $K$ is $\cal D$-reflexive. Since the subcategory of $\cal D$-reflexive modules is thick, also $N$ is $\cal D$-reflexive. Finally, let $N$ be an arbitrary finitely generated left $S$-module. Then $\im\omega_N\cong N/{\Rej_U(N)}$ is finitely generated torsionfree, hence $\cal D$-reflexive by the previous argument. Thus, $\im\omega_N$ is a finitely generated submodule of the finitely presented module $\Delta^{2}(N)$, hence it is finitely presented as well. It follows that $\Rej_U(N)$ is finitely generated and torsion, whence $\cal D$-reflexive, and $N$ is $\cal D$-reflexive. \qedhere 
\end{proof}

Within the setting and notation of the previous proposition, we see that the asymmetry between the ring structure of $R$ and $S$ also affects the $\cal D$-reflexive right~$R\mathchar`-$ and left $S$-modules.

\begin{prop}
Let ${}_SU_R$ be a cotilting bimodule.
\begin{enumerate}
\item[(i)] If the class of $\cal D$-reflexive right $R$-modules coincides with the subcategory of finitely generated right $R$-modules, then $R$ is right noetherian.

\item[(ii)] If both the classes of $\cal D$-reflexive modules coincide with the corresponding classes of finitely generated modules, then $U_R$ and ${}_SU$ are $\Sigma$-pure-injective.
\end{enumerate}
\end{prop}
\begin{proof}
\noindent(i) Let $0\to K\to R_R\to M\to0$ be a short exact sequence in $\rMod R$. Since $R_R$ and $M$ are both $\cal D$-reflexive, then also $K$ is, hence it is finitely generated by hypothesis.

\noindent(ii) Since by part~(i) the ring $R$ is right noetherian, by \cite[Theorem~6.3]{CMT10} in order to prove that $U_R$ is $\Sigma$-pure-injective it suffices to show that \cite[Condition~5.1]{CMT10} holds true for the torsionfree class ${}^{\bot_1}U_R$. Hence, let $0\to Y_0\to Y\to X\to0$ be a short exact sequence in $\rMod R$, where $Y_0,Y\in{}^{\bot_1}U$ and $Y_0$ is finitely generated, and $X\in{}^{\bot_0}U$. By hypothesis and by \cite[Lemma~3.1\,(1,a)]{MT14}, $\Delta(Y_0)$ is finitely generated hence a $\cal D$-reflexive left $S$-module. Thus, $\Delta\Gamma(X)=0$ and $H^0(\what\omega_X)\colon X\to\Gamma^2(X)$ is a monomorphism, in view of the following commutative diagram with exact rows,
\[
\xymatrix{%
	& 0 \ar[r] & Y_0 \ar[r]\ar[d]_\cong &
		Y \ar[r]\ar@{>->}[d] & X \ar[r]\ar[d]^-{H^0(\what\omega_X)} & 0 \\
	0 \ar[r] & \Delta\Gamma(X) \ar[r] & \Delta^{2}(Y_0) \ar[r] & \Delta^{2}(Y) \ar[r] &
		\Gamma^2(X) \ar[r] & 0
}	
\]
Moreover, since $\Delta(Y_0)\to\Gamma(X)$ is an epimorphism, also $\Gamma(X)$ is finitely generated hence $\cal D$-reflexive, and torsion. Since $\mathbf{R}\Delta(\what\omega_{\Gamma(X)})\circ\what\omega_{\mathbf{R}\Delta(\Gamma(X))}=\id_{\mathbf{R}\Delta(\Gamma(X))}$, also $\mathbf{R}\Delta(\Gamma(X))=\Gamma^2(X)[-1]$ is $\cal D$-reflexive. Therefore, since $X\lhook\joinrel\to\Gamma^2(X)$ and $R$ is right noetherian, $X$ is finitely generated, i.e.\ $\cal D$-reflexive, so by extension-closure also $Y$ is finitely generated. Same argument for ${}_SU$. \qedhere
\end{proof}

We are now in a position to prove the necessity of both the rings being one-sided artinian for the relevant $\De$-reflexive modules to be the finitely generated ones, hence to ensure that a fundamental feature of Azumaya duality is preserved in any cotilting duality.

\begin{thm}
Let ${}_SU_R$ be a cotilting bimodule. The following are equivalent:%
\label{t:CotiltingBimoduleBothArtinianRings}%
\begin{enumerate}
\item[(a)] $R$ is right artinian and $S$ is left artinian;

\item[(b)] the $\cal D$-reflexive modules either over $R$ and $S$ are the finitely generated ones.
\end{enumerate}

\noindent If one of these conditions holds true, then $U_R$ and ${}_SU$ are product~complete.
\end{thm}
\begin{proof}
\noindent``${\rm(a)\Rightarrow(b)}$'' It follows by \cite[Proposition~3.6]{MT14}.

\noindent``${\rm(b)\Rightarrow(a)}$'' By (the proof of) the previous proposition, both the rings $R$ and $S$ are noetherian and $U_R$ and ${}_SU$ are $\Sigma$-pure-injective. Moreover, since $U_R$ and ${}_SU$ are finitely generated being $\cal D$-reflexive, they are noetherian modules. By \cite[Lemma~4.3]{KS97}, since $U_R$ and ${}_SU$ are $\Sigma$-pure-injective and noetherian over their endomorphism ring, they are endofinite. In particular, ${}_SU$ and $U_R$ have finite length. Now, applying $\Delta_R$ on the epimorphism $R^n\to U_R\to0$ ($n\in\N$) we obtain the monomorphism ${}_SS\lhook\joinrel\to{}_SU^n$, hence $S$ is left artinian being finitely $U$-cogenerated. Analogously one proves that $R$ is right artinian.

Finally, since $U_R$ and ${}_SU$ are of finite length and $\Sigma$-pure injective, again by \cite[Lemma~4.3]{KS97} they are endofinite hence product~complete. \qedhere
\end{proof}

Summing up the results we proved so far on a cotilting duality over two artinian rings, we have the following:

\begin{cor}
Let $R$ be a right artinian ring, $S$ a left artinian ring, and ${}_SU_R$ a cotilting bimodule. Then there is a duality%
\label{c:CotiltingBimoduleBothArtinianRings}
\[
	\mathbf{R}\Delta_S:\De^b_{S\lmod}(S^\textup{op})\longrightleftarrows
		\De^b_{\rmod R}(R):\mathbf{R}\Delta_R.
\]
\end{cor}

Altogether, we get a generalisation of Azumaya's Theorem~\ref{t:Azumaya1}. Recall that, if $R$ and $S$ are one-sided artinian as in the previous result, then $\rmod R$ and $S\lmod$ are exact abelian subcategories of $\rMod R$ and $S\lMod$, respectively, and $\De^b_{\rmod R}(R)$ and $\De^b_{S\lmod}(S^\textup{op})$ are triangle equivalent to $\De^b_{\vphantom R}(\rmod R)$ and $\De^b_{\vphantom S}(S\lmod)$, respectively (see \cite[Proposition~4.8]{Har66}).

Let us pass to discuss the representability of a derived duality by a cotilting bimodule, namely how to generalise Azumaya's Theorem~\ref{t:Azumaya2}. Assume that $F:S\lMod\leftrightarrows\rMod R:G$ is a right adjoint pair of contravariant functors, and consider the right adjunction
\[
	\mathbf{R}F:\De^b(S^{\textup{op}})\longrightleftarrows
	\De^b(R):\mathbf{R}G.
\]
Within this generality, the resulting ($\De$-)reflexive complexes are those $M^\chb\in\De^b(R)$ and $N^\chb\in\De^b(S^\textup{op})$ such that the maps $M^\chb\to\mathbf{R}F\circ\mathbf{R}G(M^\chb)$ and $N^\chb\to\mathbf{R}G\circ\mathbf{R}F(N^\chb)$ are quasi-isomorphisms.
In order to generalise Theorem~\ref{t:Azumaya2}, it is natural to ask when the displayed derived adjoint pair induces a duality either between the bounded derived subcategories of complexes with finitely generated cohomologies,
\[
	\De^b_{S\lmod}(S^{\textup{op}})\longrightleftarrows
	\De^b_{\rmod R}(R),
\]
or between the bounded derived categories of the finitely generated modules,
\[
	\De^b(S\lmod)\longrightleftarrows
	\De^b(\rmod R).
\]
The previous four derived subcategories suffer existential issues; nonetheless, in the spirit of Theorem~\ref{t:Azumaya2} this shall be thought as the analogue of $\rmod R$ and $S\lmod$ being not thick subcategories in general. In this regard, on the one hand \cite[Corollary~3.6]{MT10} ensures that if $F$ and $G$ have cohomological dimension at most~$1$ and map projectives into $G$-acyclic and $F$-acyclic objects, respectively, then the $\De$-reflexive complexes are precisely the bounded complexes with $\De$-reflexive cohomology (notice that the first assumption makes the second equivalent to $\mathbf{R}G\circ\mathbf{R}F\cong \mathbf{L}(GF)$ and $\mathbf{R}F\circ\mathbf{R}G\cong \mathbf{L}(FG)$ naturally, see \cite[Proposition~5.4]{Har66}). Since the derived categories $\De^b_{\mathcal{M}_R}(R)$ and $\De^b_{{}_S\mathcal{M}}(S^\textup{op})$ are well defined, the duality we shall represent does appear as in the former form. On the other hand, Corollary~\ref{c:CotiltingBimoduleBothArtinianRings} suggests that these hypotheses must be taken into account, in order to get, besides representability, one-sided artinian rings. With these premises, we have the following:

\begin{thm}
Let $F:S\lMod\rightleftarrows\rMod R:G$ be two additive contravariant functors such that
\[
	\mathbf{R}F:\De^b_{S\lmod}(S^{\textup{op}})\longrightleftarrows
	\De^b_{\rmod R}(R):\mathbf{R}G
\]
is a duality. Assume that the cohomological dimension of $F$ and $G$ is at most~$1$ and that $F$ and $G$ map projectives into $G$-acyclic and $F$-acyclic objects, respectively. Then%
\label{t:Representability}%
\begin{enumerate}
\item[(i)] The duality is representable by a faithfully balanced bimodule ${}_SU_R$ which is finitely generated on both sides;

\item[(ii)] ${}_SU_R$ is a cotilting bimodule;

\item[(iii)] $R$ is right artinian and $S$ is left artinian.

\item[(iv)] The $\cal D$-reflexive modules coincide with the finitely generated ones.
\end{enumerate}
\end{thm}
\begin{proof}
First notice that $R$ turns out to be right noetherian and $S$ left noetherian. Indeed, let $I\le R$ be a right ideal of $R$. By assumption, any finitely generated right $R$-module is $\cal D$-reflexive, so in the exact triangle $I\to R_R\to \smash[t]{R/I\buildrel+\over\to}$ of $\De^b(R)$ both $R_R$ and $R/I$ are finitely generated, hence $\cal D$-reflexive, so that $I$ is $\cal D$-reflexive as well. In turn, we obtain that $\mathbf{R}G(I)$ belongs to $\De^b_{S\lmod}(S^\textup{op})$, and therefore $\mathbf{R}FG(I)\cong I$ belongs to $\De^b_{\rmod R}(R)$; that is, $I=R^0F(G(I))$ is finitely generated (as usual, $R^\ast F$ is a shorthand for $H^\ast\mathbf{R}F$). Similar argument to show that $S$ is left noetherian. Thus, the categories $\De^b_{\rmod R}(R)$ and $\De^b_{S\lmod}(S^{\textup{op}})$ are triangulated and equivalent respectively to $\De^b(\rmod R)$ and $\De^b(S\lmod)$ \cite[Proposition~4.8]{Har66}. From now on we will deal with these latter derived categories.

\noindent(i) Since $R_R$ and ${}_SS$ are projective, their stalks are acyclic with respect to any additive functor, hence $R^nG(R_R)=0=R^nF({}_SS)$ for all $n\ne0$; that is, $\mathbf{R}G(R_R)$ and $\mathbf{R}F({}_SS)$ are stalk complexes concentrated in degree zero. Set ${}_SV=\mathbf{R}G(R_R)$ and $U_R=\mathbf{R}F({}_SS)$; notice that by hypothesis $U_R$ and ${}_SV$ are finitely generated. As usual, $U_R$ admits a natural structure of left $S$-module compatible with the $R$-module structure; analogously, $V$ is a $S$-$R$-bimodule. Let us show that the bimodules ${}_SU_R$ and ${}_SV_R$ are isomorphic. By the (right) adjunction of $\mathbf{R}F$ and $\mathbf{R}G$, for every $N^\chb\in\De^b(S\lmod)$ and every $M^\chb\in\De^b(\rmod R)$ we have
\[
	\Hom_{\De^b(\rmod R)}(M^\chb,\mathbf{R}F(N^\chb))\cong
	\Hom_{\De^b(S\lmod)}(N^\chb,\mathbf{R}G(M^\chb)),
\]
so the claim readily follows by choosing $M^\chb=R_R$ and $N^\chb={}_SS$ \cite{Ver96}. Moreover, the bimodule ${}_SU_R$ is faithfully balanced since
\begin{align*}
	R\cong \Hom_R(R,R) &=\Hom_{\De^b(\rmod R)}(R,R) \cr
	&\cong \Hom_{\De^b(\rmod R)}(R,\mathbf{R}F\circ\mathbf{R}G(R_R)) \cr
	&\cong \Hom_{\De^b(S\lmod)}({}_SU,\mathbf{R}G(R)) \cr
	&= \End_{\De^b(S\lmod)}(U)=\End_S(U),
\end{align*}
and analogously for the other ring isomorphism. Now, let us prove that the duality is representable by ${}_SU_R$; that is, there are natural isomorphisms $\mathbf{R}F\cong\mathbf{R}\Delta_S$ and $\mathbf{R}G\cong\mathbf{R}\Delta_R$. For instance, for any $M^\chb\in\De^b(\rmod R)$, the complex $\mathbf{R}G(M^\chb)$ is computed as the complex $0\to G(P_0)\to G(P_1)\to G(P_2)\to\cdots$ for a homotopically projective resolution $P_\chb\to M^\chb$ formed by finitely generated projective terms, and for every $n\ge0$ we have
\begin{align*}
	G(P_n) &\cong \Hom_S(S,G(P_n))\cong \Hom_{\De^b(S\lmod)}(S,\mathbf{R}G(P_n)) \cr
	& \cong \Hom_{\De^b(\rmod R)}(\mathbf{R}F\circ\mathbf{R}G(P_n),\mathbf{R}F({}_SS))\cr
	& \cong \Hom_{\De^b(\rmod R)}(P_n,U_R)=\Delta_R(P_n)
\end{align*}
naturally; that is, the complex $\mathbf{R}G(M^\chb)$ is naturally isomorphic to $\Delta_R(P_\chb)$, i.e.~to $\mathbf{R}\Delta_R(M^\chb)$, as desired. A similar argument shows that $\mathbf{R}F\cong\mathbf{R}\Delta_S$ naturally.

\noindent(ii) Notice that the hypotheses on the functors $\Delta_R$ and $\Delta_S$ imply that the injective dimension both of ${}_SU$ and $U_R$ is at most~$1$. Moreover, for any cardinal $\alpha$, we have $U_R^{\alpha}=\Delta_S(S^{(\alpha)})$, so it is $\Delta_R$-acyclic, i.e.\ $\Gamma_{\!R}(U^\alpha)=0$. Similarly, $\Gamma_{\!S}(U^\alpha)=0$ for any cardinal $\alpha$. Thus, the bimodule $_SU_R$ is partial cotilting on both sides. To conclude that it is a cotilting bimodule, let us prove that ${}^{\bot_0}U\cap{}^{\bot_1}U=0$ either in $S\lMod$ and $\rMod R$. For instance, let $M\in{}^{\bot_0}U\cap{}^{\bot_1}U$ be a nonzero right $R$-module. Given a finitely generated submodule $M'\le M$; from the short exact sequence $0\to M'\to M\to M''\to0$ we get the triangle $\smash[t]{\mathbf{R}\Delta(M'')\to\mathbf{R}\Delta(M)\to\mathbf{R}\Delta(M')\buildrel+\over\to}$, hence
\[
	\Delta(M')=R^0\Delta(M')\cong R^1\Delta(M'')=\Gamma(M'')
\]
by hypothesis. Since $\Delta_R(M')\in{}^{\bot_1}U$, we obtain
\[
	0=\Ext^1_S(\Gamma(M''),U)^\alpha\cong\Ext^1_S(\Gamma(M''),U^\alpha)
\]
for every cardinal $\alpha$. This means that from the short exact sequence $0\to K\buildrel i\over\to R^{(\alpha)}\to M''\to0$ in $\rMod R$ we obtain the split exact sequence $0\to U^\alpha\to \Delta(K)\to\Gamma(M'')\to0$ in $S\lMod$, whence $\Delta^{2}(K)\cong \Delta^{2}(R^{(\alpha)})\oplus\Delta\Gamma(M'')$. Therefore, from these facts together with the hypothesis, we have that the morphism $\what\omega_{M''}$ is described by the following cochain map:
\[
\xymatrix{%
	M'' \ar[d]^{\what\omega_{M''}} \\
	\mathbf{R}\Delta^{2}(M'')
}
	\hskip1truecm
\xymatrix{%
	0 \ar[r] & K \ar[d]_{\omega_K} \ar[r]^i &
		R^{(\alpha)}\ar[d]^{\omega_{R^{(\alpha)}}} \ar[r] & 0 \\
	0 \ar[r] & \Delta^{2}(K) \ar[r]^{\Delta^{2}(i)} & \Delta^{2}(R^{(\alpha)}) \ar[r] & 0
}	
\]
and since $\Delta^{2}(i)$ is a split epimorphism, $\what\omega_{M''}$ is null~homotopic. In particular, $\mathbf{R}\Delta(\what\omega_{M''})=0$, contradiction to the adjunction formula $\mathbf{R}\Delta(\what\omega_{M''})\circ\what\omega_{\mathbf{R}\Delta(M'')}=\id_{\De^b(S^\textup{op})}$.

\noindent(iii)\enspace  First let us show that $_SU$ and $U_R$ are $\Sigma$-pure-injective. Since the rings are  neotherian, by \cite[Theorem 6.3]{CMT10} we just have to show that if $0\to Y_0\to N\to X\to 0$ is exact, with $Y_0$ finitely generated, $N\in  \Cogen U$ and $\Delta(X)=0$, then $N$ is finitely generated.
By the inclusion  $0\to \Delta(N)\to \Delta(Y_0)$, we have that  $\Delta(N)$ and hence $\Delta^2(N)$ are finitely generated. Since $N\lhook\joinrel\to \Delta^2(N)$, it is finitely generated as well. By Remark~\ref{r:FinitenessConditions}, ${}_SU$ and $U_R$ are endofinite. Since $S$ and  $R$ are finitely cogenerated by ${}_SU$ and $U_R$, respectively, they are left artinian and right artinian, respectively.

\noindent(iv)\enspace It follows at once by Theorem~\ref{t:CotiltingBimoduleBothArtinianRings}.
\end{proof}

\begin{rem}
If $F:S\lMod\rightleftarrows\rMod R:G$ are two additive contravariant functors such that
\[
	\mathbf{R}F:\De^b_{S\lmod}(S^\textup{op})\longrightleftarrows
	\De^b_{\rmod R}(R):\mathbf{R}G
\]
is a duality and their cohomological dimension is at most~$1$, then the first part of the previous proof still holds, so that the duality is representable by a faithfully bimodule ${}_SU_R$ which is finitely generated on both sides. Assume now $S$ right perfect and $R$ left perfect. By Remark~\ref{r:FinitenessConditions}, the finitely generated module $U_R$ is product complete since $S$ is left coherent and right perfect, and $U_R$ is finitely presented. Similar argument to get that ${}_SU$ is product complete and finitely presented. In particular, $\Gamma(U^{\alpha})=0$ since $U^{\alpha}\in \Add U$, for all cardinal $\alpha$, whence $F$ and $G$ map projectives into $G$-acyclic and $F$-acyclic objects, respectively. By the previous theorem, we infer that $R$ and $S$ are right and left artinian, respectively.
\end{rem}

\end{document}